\documentclass[11pt]{amsart}
\usepackage{graphicx}
\usepackage[active]{srcltx}
 \makeatletter
\renewcommand*\subjclass[2][2000]{%
  \def\@subjclass{#2}%
  \@ifundefined{subjclassname@#1}{%
    \ClassWarning{\@classname}{Unknown edition (#1) of Mathematics
      Subject Classification; using '1991'.}%
  }{%
    \@xp\let\@xp\subjclassname\csname subjclassname@#1\endcsname
  }%
}
 \makeatother
\usepackage{enumerate,url,amssymb,  mathrsfs}

\newtheorem{theorem}{Theorem}[section]
\newtheorem{lemma}[theorem]{Lemma}
\newtheorem*{lemma*}{Lemma}
\newtheorem{proposition}[theorem]{Proposition}
\newtheorem{corollary}[theorem]{Corollary}

\theoremstyle{definition}

\newtheorem{example}[theorem]{Example}

\theoremstyle{remark}
\newtheorem{remark}[theorem]{Remark}

\numberwithin{equation}{section}

\newcommand{\abs}[1]{\lvert#1\rvert}

\newcommand{\R}{\mathbb{R}}

\DeclareMathOperator{\sign}{sign}

 \DeclareMathOperator{\Res}{Res}

\def\XXint#1#2#3{{\setbox0=\hbox{$#1{#2#3}{\int}$}
\vcenter{\hbox{$#2#3$}}\kern-.5\wd0}}

\def\le{\leqslant}
\def\ge{\geqslant}
\begin{document}

\title{Cauchy transform and Poisson's equation}
\subjclass{Primary 35J05; Secondary 47G10}


\keywords{M\"obius transformations, Poisson equation, Newtonian
potential, Cauchy transform, Bessel function}
\author{David Kalaj}
\address{University of Montenegro, Faculty of Natural Sciences and
Mathematics, Cetinjski put b.b. 81000 Podgorica, Montenegro}
\email{davidk@t-com.me}

\begin{abstract}
Let $u\in W^{2,p}_0$, $1\le p\le \infty$ be a solution of the
Poisson equation $\Delta u = h$, $h\in L^p$, in the unit disk. It is
proved that  $\|\nabla u\|_{L^p} \le a_p\|h\|_{L^p}$ with sharp
constant $a_p$ for $p=1$ and $p=\infty$ and that $\|\partial
u\|_{L^p} \le b_p\|h\|_{L^p}$ with sharp constant $b_p$ for $p=1$,
$p=2$ and $p=\infty$. In addition is proved that for $p>2$
$||\partial u||_{L^\infty}\le c_p\Vert h\Vert_{L^p} $, and $||\nabla
u||_{L^\infty}\le C_p\Vert h\Vert_{L^p}, $ with sharp constants
$c_p$ and $C_p$. An extension to smooth Jordan domains is given.
These problems are equivalent to determining the precise value of
$L^p$ norm of {\it Cauchy transform of Dirichlet's problem}.
\end{abstract} \maketitle

\tableofcontents

\section{Introduction}

\subsection{Notation}
By $\mathbf{U}$ is denoted the unit disk in the complex plane and by
$\mathbf{T}$ its boundary. By $\Omega$ is denoted a bounded domain
in complex plane. By $$dA(z)=dx dy,\ \  z=x+iy,$$ is denoted the
Lebesgue area measure in the unit disk and by
$$d\mu(z) = \frac{1}{\pi}dx dy
$$ is denoted normalized area measure. Here $W^{k,p}(\Omega)$ is the Banach space of
$k$-times weak differentiable $p-$integrable functions. The norm in
$W^{k,p}(\Omega)$ is defined by
$$\|u\|_{W^{k,p}}:=\left(\int_\Omega\sum_{|\alpha|\le k}|D^\alpha
u|^p d\mu\right)^{1/p},$$ where $\alpha\in\Bbb N_0^2$. If $k=0$,
then $W^{k,p}=L^p$ and instead of $\|u\|_{L^{p}}$ we sometimes write
$\|u\|_{p}$. Another Banach space $W^{k,p}_0(\Omega)$ arises by
taking the closure of $C^{k}_0(\Omega)$ in $W^{k,p}(\Omega)$ (here
$C^{k,p}_0(\Omega)$ is the space of $k$ times continuously
differentiable functions with compact support in $\Omega$,
\cite[p.~153-154]{gt}).

The main subject of this paper is a weak solution of Dirichlet
problem  \begin{equation}\label{poi}\left\{
           \begin{array}{ll}
             u_{z\bar z} = g(z), & z\in \Omega \\
             u\in W^{1,p}_0({ \Omega})&
           \end{array}
         \right.\end{equation}
where $4u_{z\bar z}=\Delta u$ is the Laplacian. This equation is the
Poisson's equation. A weak differentiable function $u$ defined in a
domain $\Omega$ with $u\in W^{2,p}_0(\Omega )$ is a weak solution of
Poisson's equation if $D_1 u$ and $D_2 u$ are locally integrable in
$\Omega $ and
$$\int_{\Omega} (D_1 u D_1 v + D_2 u D_2 v + 4g v) d\mu(z) = 0,$$ for
all $v\in C^1_0(\Omega)$.

It is well known that for $g\in L^p(\Omega)$, $p>1$, the weak
solution $u$ of Poisson's equation is given explicitly as the sum of
Newtonian potential $$N(g) =\frac{2}{\pi} \int_{\Omega } \log|z-w|
g(w)dudv,\ \ w=u+iv$$ and a harmonic function $h$ such that
$h|_{\partial \Omega} + N(g)|_{\partial \Omega}\equiv u|_{\partial
\Omega}$. In particular if  $\Omega=\mathbf{U}$, then the function
\begin{equation}\label{sol}u(z) = \frac{2}{\pi}\int_{\mathbf{U}} \log
\frac{|z-w|}{|1-\overline zw|}g(w)dA(w)\end{equation} is the
explicit solution of \eqref{poi}.

The function $G$ given by
\begin{equation}\label{green1}
G(z,w)=\frac{2}{\pi}\log\left|\frac{z-w}{1-z\overline w}\right|,\
z,w\in \mathbf U,
\end{equation}
is called the Green function of the unit disk $\mathbf{U}\subset\Bbb
C$ w.r. to Laplace operator.
For $g\in L^p(\mathbf{U})$, $p>1$, and $$u(z) =
\frac{2}{\pi}\int_{\mathbf{U}}  \log \frac{|z-\omega|}{|1-z\overline
\omega|} g(\omega)dA(\omega),
$$ the {\it Cauchy transform and conjugate Cauchy transform for Dirichlet's problem} (see
\cite[p.~155]{aim}) of $g$ are defined by
\begin{equation}\label{ca1}\mathcal{C}_{\mathbf{U}}[g](z) = \frac{\partial u}{\partial z} =
\frac{1}{\pi}\int_{\mathbf{U}} \frac{1-|\omega|^2}{(\omega-z
)(\bar\omega z-1)}  g(\omega)dA(\omega)\end{equation} and
\begin{equation}\label{ca2}\mathcal{\bar C}_{\mathbf{U}}[g](z) = \frac{\partial u}{\partial
\bar z} = \frac{1}{\pi}\int_{\mathbf{U}} \frac{1-|\omega|^2}{(\bar
\omega -\bar z )(\omega \bar z-1)}
g(\omega)dA(\omega).\end{equation} Here we use the notation
$$\frac{\partial }{\partial  z}:=\frac 12\left(\frac{\partial
}{\partial x}+ \frac 1i \frac{\partial }{\partial  y}\right)\text{
and }\frac{\partial }{\partial \bar z}:=\frac 12\left(\frac{\partial
}{\partial x}- \frac 1i \frac{\partial }{\partial  y}\right).$$ It
is well-known that for $p>1$ Cauchy transforms
$$\mathcal{C}_{\mathbf{U}} \colon L^p(\mathbf U)\to L^p(\mathbf U)\ \text{ and }
\mathcal{\bar C}_{\mathbf{U}}\colon L^p(\mathbf U)\to L^p(\mathbf
U)$$ are bounded operators. Recall that the norm of an operator
$T:X\to Y$ between normed spaces $X$ and $Y$ is defined by
$$\|T\|_{X\rightarrow Y}=\sup\{\|T x\|: \|x\|=1\}.$$
The Jacobian matrix of a mapping $u:\Bbb C\to \Bbb C$ is defined by
$$\nabla u = \left(
                                    \begin{array}{cc}
                                      D_1u_1 & D_2 u_1 \\

D_1 u_2 & D_2 u_2 \\                                    \end{array}
\right).$$ The matrix $\nabla u $ is given by
\begin{equation}\label{e:POISSON1}\begin{aligned}\nabla
u(z)h=\frac{2} {\pi} \int_{\mathbf{U}}\left<\frac{(1-|\omega|^2)}{(
\omega-z)(z \bar \omega -1)}, h\right>\, g(\omega)\,dA(\omega),\ \
h\in \Bbb C.\end{aligned}\end{equation} Here $\left<\cdot,\cdot
\right>$ denotes the scalar product. The equation \eqref{e:POISSON1}
defines {\it the differential operator of Dirichlet's problem}
$$\mathcal{D}_{\mathbf{U}}: L^p(\mathbf{U}, \Bbb C) \to L^p(\mathbf{U},
\mathcal{M}_{2,2}), \ \mathcal{D}_{\mathbf{U}}[g] = \nabla u.$$ Here
$\mathcal{M}_{2,2}$ is the space of square $2\times 2$ matrices $A$
by the induced norm: $|A|=\max\{|A h|: |h|=1\}$.

With respect to the induced norm there holds
\begin{equation}\label{qc1}|\nabla u|= \abs{\partial u}+\abs{\bar
\partial u},\end{equation} and this implies that \begin{equation}\label{qc}|\mathcal{D}_{\mathbf{U}}[g]|=|\mathcal{C}_{\mathbf{U}}[g]|+|\mathcal{\bar
C}_{\mathbf{U}}[g]|.\end{equation}


\subsection{Background}
The starting point of this paper is the celebrated Calderon-Zygmund
Inequality which states that. Let $g\in L^p(\Omega)$, $1<p<\infty$,
and let $w$ be the Newtonian potential of $g$. Then $u\in
W^{2,p}(\Omega)$, $\Delta u=g$ $\mathrm{a.e.}$ and
\begin{equation}\label{zy} ||D^2 u\Vert _p\le C\Vert g\Vert _p
\end{equation}
where  $D^2 u$ is the weak Hessian matrix of $u$ and $C$ depends
only on $n$ and $p$. Calderon-Zygmund Inequality  is one of the main
tools in establishing the a priory bound of $W^{2,p}$ norm of $u$ in
terms of the function $g$ and boundary condition (see
\cite[Theorem~9.13]{gt} or classical paper by Agmon, Douglis and
Nirenberg \cite{adn}). It follows from these a priory bounds that
for $p>1$ there exists a constant $C_p$, such that
\begin{equation}\label{vert}\Vert \nabla u\Vert_{p}\le C_p\|g\|_p, \ \ \text{for} \ \
u\in W^{1,p}_0(\mathbf{U}).\end{equation}

We refer to \cite[Problem~4.10, p. 72]{gt} for some related
estimates that are not sharp for the case $u\in C^2_0(\Bbb B^n)$,
where $\Bbb B^n$ is the unit ball in $\Bbb R^n$.

Suppose now that $g$ is in $L^2(\Omega)$, where $\Omega$ is a
bounded domain in the complex plane, and that $g=0$ outside
$\Omega$. The Cauchy transform $\mathfrak{C} [g]$ of $g$, is defined
by
$$\mathfrak{C} [g]= \frac{1}{\pi}\int_{\mathbf{\Omega}}
\frac{g(z)}{w-z}dA(w).$$ Similarly is defined the Cauchy transform
with respect to some positive Radon measure $\nu$. The operator
$\mathfrak{C}$ is a bounded operator from $L^2(\Omega)$ into itself.
We want to point out the following result of Anderson and Hinkkanen
\cite{ah}. If $\Omega=\mathbf U$, the Cauchy transform $\mathfrak{C}
[g]$ restricted to $\mathbf{U}$, satisfies
\begin{equation}\label{and}\|\mathfrak C{[g]}\|_2\leq
\frac{2}{\alpha}\|g\|_2,\end{equation} where $\alpha\approx 2.4048$
is the smallest positive zero of the Bessel function $ J_0:$
$$J_0(x)=\sum^\infty_{k=0}\frac{(-1)^k}{{k!}^2}\left(\frac x
2\right)^{2k}.$$ This inequality is sharp. This result has been
extended by Dostani\'c to the smooth domains with sharp constants
\cite{dost1}, and for $p\neq
2$ with some constants that are asymptotically sharp when $p$ is close to $1$ or $2$, see \cite{dost}. 

Associated to this Cauchy transform is the Beurling transform
(called also Ahlfor-Beurling transform or Hilbert transform)
$$\mathfrak{B}[g](z) = \partial_z\mathfrak{C}[g](z) = \mathrm{pv} \int_{\mathbf{U}}\frac{g (w)}{(w - z)^2}
d\mu(w),$$ where "pv" indicates the standard principal value
interpretation of the integral.

On the other hand, associated to Cauchy transform of Dirichlet's
problem is the {\it Beurling transform of Dirichlet's problem}
(\cite{aim})
$$\mathcal{S}[g](z) = \partial_z\mathcal{C}_{\Bbb U}[g](z) = \mathrm{pv} \int_{\mathbf{U}}\left(\frac{1}{(w -
z)^2}+\frac{\overline w^2}{(1-\overline wz)^2}\right)g(w) d\mu(w).$$

The Beurling transform and Beurling transform of Dirichlet's problem
are bounded operator in $L^p$, $1<p<\infty$. This follows from
Calderon-Zygmund Inequality. However, determining the precise value
of the $L^p$-norm for $p\neq 2$ of Beurling transform is a
well-known and long-standing open problem. On the other for $p=2$,
both Beurling transforms are the isometries of Hilbert space
$L^2(\mathbf U)$, and therefore have the norms equal to $1$, see
\cite[Theorem~4.8.3]{aim} and \cite[p.~87-111]{Ahl}. Beurling
transforms are important in connection with  nonlinear elliptic
system in the plane and Beltramy equation (see \cite{ais},
\cite[Chapter~V]{Ahl}, \cite[Chapter~IV]{aim}). Cauchy transform and
Cauchy transform of Dirichlet's problem are connected by
$$\mathcal C_{\mathbf{U}}[g](z)=(\mathfrak{C}-\mathfrak{J}_0^*)[g](z),$$
where
$$\mathfrak{J}_0^*[g](z)= \frac{1}{\pi}\int_{\mathbf{U}}\frac{\overline \omega}{1-z\overline\omega}g(\omega)dA(\omega),$$
which satisfies $$\mathfrak{J}_0^*= \mathfrak{B}\mathfrak{C}.$$ Thus
\begin{equation}\label{rende}
\mathcal C_{\mathbf{U}}=\mathfrak{C}-\mathfrak{B}\mathfrak{C}.
\end{equation}

The same can be repeated for conjugate Cauchy transform for
Dirichlet's problem and conjugate Beurling transform for Dirichlet's
problem. See \cite{duke} for this topic. Unlike the Beurling
transform, the Cauchy transform is not a bounded operator considered
as a mapping from $L^2(\mathbf C)$ into itself. The reason is that
the Lebesgue measure $dA(\omega)$ of the complex plane do not
satisfies {\it linear growth condition}. Let $\nu$ be a continuous
positive Radon measure on $\mathbf C$ without atoms. According to
the result of Tosla \cite{tosla}, the Cauchy integral of the measure
$\nu$ is bounded on $L^2(\mathbf C,\nu)$ if and only if $\nu$  has
linear growth and satisfies the local curvature condition.

One of primary aims of this paper is to give an explicit constant
$C_p$ of inequality \eqref{vert}, and to generalize the inequality
\eqref{and} for Cauchy transform of Dirichlet's problem, which is
equivalent with the problem of estimation of the following norms
$\|\mathcal{C}_{\mathbf{U}}\|_{L^p\to L^p}$, $\|\mathcal{\bar
C}_{\mathbf{U}}\|_{L^p\to L^p}$, and
$\|\mathcal{D}_{\mathbf{U}}\|_{L^p\to L^p}$. It follows from
\eqref{qc} and \eqref{vert} that these norms are finite and that
they can be estimated in terms of $p$. In this paper we deal with
the exact values of these norms.

The first main result of this paper is
\\
\\
{\bf Theorem~A} {\it Let $\alpha\approx 2.4048$ be the smallest
positive zero of the Bessel function $J_0$. For $1\le p\le 2$ we
have
\begin{equation}\label{10kad}\|\mathcal{C}_{\mathbf{U}}\|_{L^p\to L^p}\le
\frac{2}{\alpha^{2-2/p}} \text{ and }\|\mathcal{\bar
C}_{\mathbf{U}}\|_{L^p\to L^p}\le
\frac{2}{\alpha^{2-2/p}},\end{equation} and for $2\le p\le \infty$
we have
\begin{equation}\label{10kadtad}\|\mathcal{C}_{\mathbf{U}}\|_{L^p\to L^p}\le
\frac{4}{3}\left(\frac{3}{2\alpha}\right)^{2/p} \text{ and
}\|\mathcal{\bar C}_{\mathbf{U}}\|_{L^p\to L^p}\le
\frac{4}{3}\left(\frac{3}{2\alpha}\right)^{2/p}.\end{equation} The
equality is attained in all inequalities in \eqref{10kad} and
\eqref{10kadtad} for $p=1$, $p=2$  and $ p =\infty$. Moreover for
$1\le p\le 2$ there holds the inequality
\begin{equation}\label{10dri}\|\mathcal D_{\mathbf U}\|_{L^p\to L^p} \le 4 \alpha^{2/p-2},\end{equation} and if $2\le
p\le \infty$
\begin{equation}\label{10drili}\|\mathcal D_{\mathbf U}\|_{L^p\to L^p} \le \frac{16}{3\pi}
\left(\frac{3\pi}{4\alpha}\right)^{2/p}.\end{equation} The equality
is attained in \eqref{10dri} and \eqref{10drili} for $p=1$ and
$p=\infty$. }
\\
\\

Notice that both Cauchy transforms $\mathfrak C_{\Bbb U}$ and
$\mathcal C_{\Bbb U}$ have the same Hilbert norm (cf. inequalities
\eqref{and} and \eqref{10kad} for $p=2$). It remains an open problem
the precise determining of $L^p$ norm of $\mathcal C_{\mathbf U}$
for $1<p<2$ and $2<p<\infty$ and of  $\mathcal D_{\mathbf U}$ for
$1<p<\infty$.

For $q>2$, and $g\in L^q$ the solution $u$ of Poisson
equation~\eqref{poi} is in $C^{1,\alpha}(\Omega)$ for some
$\alpha>0$ (see for example \cite{jost}) which in particular implies
that if $K\subset \Omega$ is a compact set, then there exists a
constant $C_K$ such that $|\nabla u(z)|\le C_K$, $z\in K$.  The
condition $q>2$ is the best possible (see { Example~\ref{she}}
below). We will show that for the unit disk, or more generally for
smooth domains, the gradient of solution is globally bounded on the
domain, see Corollary~\ref{rr}.

The second main result of the paper is precise estimation of
$L^\infty$ norm of gradient which can be written in terms of
operator norms as follows.
\\
\\
{\bf Theorem~B} {\it
 For $q>2$, and $p: \frac 1p +\frac 1q=1$, there hold
the following relations
\begin{equation}\label{1shkurt}\|\mathcal{C}_{\mathbf{U}}\|_{L^q\to L^\infty}
= c_p,\end{equation} \begin{equation}\label{1mars}\|\mathcal{\bar
C}_{\mathbf{U}}\|_{L^q\to L^\infty} = c_p,\end{equation} and
\begin{equation}\label{1prill}\|\mathcal{D}_{\mathbf{U}}\|_{L^q\to L^\infty}  =
C_p,\end{equation}where $$c^p_p=\mathrm{B}(1+p,1-p/2),$$
$\mathrm{B}$ is the beta function, and $$C^p_p=\frac{2^{2-p}
\Gamma[(1 + p)/2]}{\sqrt{\pi}\Gamma[1 + p/2]}c^p_p .$$ The condition
$q>2$ is the best possible.}

Together with this section, the paper contains five other sections.
Section~2 contains some important formulas and sharp inequalities
for potential type integrals. One of the main tools for the proving
of these results are M\"obius transformations of the unit disk and
the Gauss hypergeometric function. Section~3 contains the proof of
Theorem~B together with an extension to smooth Jordan domain.
Section~4 contains the proof of a weak form of Theorem~A with exact
constants for $p=1$ and for $p=\infty$. Section~5 contains the proof
of Theorem~A for the Hilbert case, namely for $p=2$. The proof is
based on Boyd theorem (\cite[Theorem~ 1,~p.~368]{boyd}), and
involves the zeros of Bessel function. By making use of Riesz-Thorin
interpolation theorem, in Section~6 we complete the proof of
Theorem~A.

\section{Some lemmas}

 We recall the classical definition of the Gauss
hypergeometric function: $${_2F_1}(a, b; c; z) = 1 +
\sum_{n=1}^\infty\frac{(a)_n(b)_n}{ (c)_n n!} z^n,$$ where
$(d)_n=d(d+1)\cdots (d+n-1)$ is the Pochhammer symbol. The series
converges at least for complex $z \in \mathbf{U}$ and for $z\in
\mathbf{T}$, if $c>a+b$. We begin with the lemma which will be used
in two our main inequalities.
\begin{lemma}[The main technical lemma]\label{veryimportant}
If $1\le p<2$, $0\le \rho<1$ and $$I_p=2
(1-\rho^2)^{2-p}\int_0^1r^{1-p}(1-r^2)^p\frac{1+r^2\rho^2}{(1-r^2\rho^2)^3}dr,$$
then
\begin{equation}\label{essential}
 I_p =\frac{\Gamma[1+p]\Gamma[1 - \frac p2] }{{ \Gamma[2 + \frac p2] }}{_2F_1}(\frac p2 - 1; p; \frac p2 + 2;
\rho^2),
\end{equation}
where ${_2F_1}$ is Gauss hypergeometric function. Moreover $I_p$ is
decreasing in $[0,1]$ and there hold
\begin{equation}\label{qwe}I_p(0)= \frac{ \Gamma[1+p]\Gamma[1 -
\frac p2] }{\Gamma[2 + \frac p2] }\end{equation} and
\begin{equation}\label{qwer}I_p(1):=\lim_{\rho\to 1-0}I_p(\rho)=\frac{\Gamma[1+p]\Gamma[1-\frac
p2]\Gamma[3-p]}{2\Gamma[2-\frac p2]}.\end{equation}
\end{lemma}
\begin{proof}

By applying partial integration we obtain
$$I_p = \int_0^1\frac{p (1 - \rho^2)^{
 2 - p} (1 - r)^{p } r^{-p/2} (-4 r + p (1 + r)^2)}{2 (1 - r)^{2 }(1 - \rho^2
r)} dr.$$ Since
\[\begin{split}&\frac{p (1 - \rho^2)^{
 2 - p}  (-4 r + p (1 + r)^2)}{2 (1 - r)^{2 }(1 - \rho^2
r)} \\&= \frac{2(p - p^2) }{(1 - \rho^2)^{p}(1-r)}+\frac{ 2(p-p^2)(1
- \rho^2)}{(1 - \rho^2)^{p}(1-r)^2}+\frac{p (-4 \rho^2 + p (1 +
\rho^2)^2)}{2(1 - \rho^2)^{p} (1-r\rho^2)},\end{split}\] by using
the well known formulas
$$_2F_1(a,b;c;z)=\frac{\Gamma(c)}{\Gamma(b)\Gamma(c-b)}\int_0^1\frac{t^{b-1}(1-t)^{c-b-1}}{(1-tz)^a}dt,$$

$$_2F_1(a,b;c;1)=\frac{\Gamma(c)\Gamma(c-a-b)}{\Gamma(c-a)\Gamma(c-b)},$$  and $$\Gamma(1+x)=x\Gamma(x)$$ we
obtain that

\[\begin{split}\mathcal I_p &=\frac{\Gamma[1+p]\Gamma[1 - \frac p2]}{{4 \Gamma[2 + \frac p2] }}\mathcal{L}_p,\end{split}\]  where
$$\mathcal{L}_p= \frac{\left(4 - p^2-(2 p+p^2) \rho^2  +
   p (-4 \rho^2 + p (1 + \rho^2)^2) _2F_1(1; 1 - \frac p2; \frac{4 +
p}{2};
      \rho^2)\right)}{(1 - \rho^2)^{p}}.$$
By using the formula $$_2F_1(a,b;c;z)  =
(1-z)^{c-a-b}{_2F_1}(c-a,c-b;c;z)$$ we obtain
$$(1-\rho^2)^{-p}{_2F_1}(1; 1 - \frac p2; \frac{4 + p}{2};
      \rho^2) =
{_2F_1}(1+\frac p2; 1 +p; \frac{4 + p}{2};
      \rho^2).$$ Having in mind the fact
$$(1-\rho^{2})^{-p}=\sum_{n=0}^\infty
(-1)^n\binom{-p}{n}\rho^{2n}=\sum_{n=0}^\infty
\frac{(p)_n}{n!}\rho^{2n},$$ by calculating the Taylor coefficients
we obtain
\[\begin{split}\mathcal{L}_p= 4 + \sum_{n=1}^\infty
\frac{4p(p^2-4)(p)_{n}}{(2(n-1)
+p)(2n+p)(2(n+1)+p)n!}\rho^{2n},\end{split}\] where
$$(p)_{n}:=\prod_{k=0}^{n-1}(p+k).$$ It follows that
\[\begin{split}\mathcal{L}_p&=4\sum_{n=0}^\infty
\frac{(\frac p2 - 1)_n(p)_{n}}{(\frac p2 + 2)_n}\rho^{2n}\\&=
4\,{_2F_1}(\frac p2 - 1; p; \frac p2 + 2; \rho^2).\end{split}\] From
$$\mathcal{L}_p=4 {_2F_1}(\frac p2 - 1; p; \frac p2 + 2; \rho^2),$$
because $\frac p2 + 2> \frac p2 - 1 +p$, we obtain
\[\begin{split}\lim_{\rho\to 1} I_p(\rho)
&=\frac{\Gamma[1+p]\Gamma[1 - \frac p2] }{{ \Gamma[2 + \frac p2]
}}{_2F_1}(\frac p2 - 1; p; \frac p2 + 2;
1)\\&=\frac{\Gamma[1+p]\Gamma[1-\frac
p2]\Gamma[3-p]}{2\Gamma[2-\frac p2]}.\end{split}\]

\end{proof}

\begin{lemma}\label{film} For
$$I_{p}(z): =\int_{\mathbb{U}}\left(\frac{1-|\omega|^2}{|z -
\omega|\cdot|1-\bar z\omega|} \right)^p\,d\mu(\omega), \ \ 1\le
p<2$$ there holds the sharp inequality
\begin{equation}\label{tr}I_p(z)\le I_p(0)=\mathrm{B}(1+p,1-p/2),\end{equation} where $\mathrm{B}$ is the beta
function. Moreover\begin{equation}\label{petar} I_1(z)=\frac{4
{_2F_1}(-1/2; 1; 5/2; |z|^2)}{3}\le \frac{4}{3}.\end{equation}
\end{lemma}
The case $p=1$ of \eqref{tr} has been already established in
\cite[Lemma~2.3]{trans}.
\begin{proof}  For a fixed $z$, we introduce the change of variables
\[\begin{aligned} \frac{z-\omega}{1-\bar z\omega}=a,  \end{aligned} \]
or, what is the same,

\[\begin{aligned} \omega=\frac{z-a}{1-\bar za}.  \end{aligned} \]
Then
\begin{align*} I_p&=\int_{\mathbb{U}}\left(\frac{1-|\omega|^2}{|z-\omega|\cdot|1-\bar z\omega|}\,
\right)^p d\mu(\omega)\\[2ex]&= \int_{\mathbb{U}}\left(\frac{1-|\omega|^2}{|a|\cdot |1-\bar
z\omega|^2}\right)^p\,d\mu(\omega)\\[2ex]& =\int_{\mathbb{U}}\left(\frac{1-|\omega|^2}{|a|\cdot |1-\bar
z\omega|^2}\right)^p\frac{(1-|z|^2)^2}{|1-\bar za|^4}\,d\mu(a)\\[2ex]& =
\int_{\mathbb{U}}
\frac{(1-|a|^2)^p(1-|z|^2)^{2+p}}{|a|^p\cdot|1-\bar z
a|^{4+2p}\,|1-\bar
z\omega|^{2p}}\,d\mu(a).\\[2ex]&
 \end{align*}
Since
 \[\begin{aligned} 1-\bar z\omega&=1-\bar z \frac{z-a}{1-\bar z a}\\&
 =\frac{1-|z|^2}{1-\bar za},
   \end{aligned} \]
by using polar coordinates, we see that
\[\begin{split}I_p&= (1-|z|^2)^{2-p} \int_{\mathbf{U}}
\frac{(1-|a|^2)^p}{|a|^p|1-\overline z a|^4} d\mu (a)\\&= \frac
1{\pi}(1-|z|^2)^{2-p}\int_0^1\rho^{1-p}(1-\rho^2)^{p}\,d\rho\int_0^{2\pi}
|1-\bar z\rho e^{i\varphi}|^{-4}\,d\varphi.\end{split}\]

 By Parseval's formula (see \cite[Theorem~10.22]{rudin}), we get

   \[\begin{aligned} \frac 1{\pi}\int_0^{2\pi}\frac{dt}{|1-\bar z \rho
e^{it}|^4}&=\frac 1{\pi}\int_0^{2\pi}\frac{dt}{|(1-\bar z \rho
e^{it})^2|^2}\\&=\frac
1{\pi}\int_0^{2\pi}\Big|\sum_{n=0}^\infty(n+1)(\bar z\rho)^n
e^{nit}\Big|^2\,dt\\&=2\sum_{n=0}^\infty(n+1)^2|z|^{2n}{\rho^{2n}}.
\end{aligned} \]
Thus $$I_p =
2(1-|z|^2)^{2-p}\sum_{n=0}^\infty(n+1)^2|z|^{2n}\int_0^1\rho^{1-p}(1-\rho^2)^{p}{\rho^{2n}}d\rho,$$
which can be written in closed form as

$$I_p =2(1-|z|^2)^{2-p}\int_0^1r^{1-p}(1-r^2)^p\frac{1+r^2\rho^2}{(1-r^2\rho^2)^3}dr.$$
From Lemma~\ref{veryimportant} it follows that
$$I_p(z)\le I_p(0) = \frac{\Gamma[1+p] \Gamma(1 - p/2)}{
\Gamma(2 + p/2)}=\mathrm{B}(1+p,1-p/2),$$ where $\mathrm{B}$ is the
beta function.

\end{proof}

\begin{lemma}\label{le}
For $\varphi \in [0,2\pi]$, $z=e^{i\alpha}\rho\in \mathbf{U}$, $1\le
p<2$ and
\begin{equation}\label{important}
\mathcal
I_p:=\frac{2}{\pi}\int_{\mathbf{U}}\left|\mathrm{Re}\frac{e^{-i\varphi}(1-|\omega|^2)}{(
\omega-z)(z \bar \omega -1)}\right|^p\,dA(\omega)\end{equation}
there holds the equality \begin{equation}\label{ess}\mathcal I_p=
\frac{2\Gamma[1+p]\Gamma[1 - \frac p2] \Gamma[\frac{1 + p}{
  2}]}{{\sqrt{\pi}\Gamma[1 + \frac p2] \Gamma[2 + \frac p2] }}{_2F_1}(\frac p2 - 1; p; \frac p2 + 2;
\rho^2).\end{equation}

Moreover
$$\frac{\Gamma[1+p]\Gamma[1-\frac
p2]\Gamma[\frac{1+p}{2}]\Gamma[3-p]}{\sqrt{\pi}\Gamma[2-\frac
p2]\Gamma[1+\frac p2]}\le \mathcal I_p \le
\frac{2\Gamma[1+p]\Gamma[1 - \frac p2] \Gamma[\frac{1 + p}{
  2}]}{{\sqrt{\pi}\Gamma[1 + \frac p2] \Gamma[2 + \frac p2] }}.$$

In particular \begin{equation}\label{graso}\mathcal I_1=\frac{16
{_2F_1}(-1/2; 1; 5/2; \rho^2)}{3\pi}\le
\frac{16}{3\pi}.\end{equation}

\end{lemma}
\begin{proof}
 For  $\varphi\in[0,2\pi]$ let $\zeta = z
e^{i\varphi}$. Then by introducing the change $w=e^{i\varphi}\omega$
we obtain
\[\begin{split}\frac{2} {\pi} \int_{\mathbf{U}}&\abs{\mathrm{Re}\frac{(1-|\omega|^2)}{(
e^{i\varphi}\omega-\zeta)(\zeta \overline{ e^{i\varphi}\omega}
-1)}}^p\,dA(\omega)\\&=\frac{2} {\pi}
\int_{\mathbf{U}}\abs{\mathrm{Re}\frac{(1-|\omega|^2)}{(w-\zeta)(\zeta
\bar w -1)}}^p\,dA(w).\end{split}\] It follows in particular that
$$\mathcal I_p = \frac{2} {\pi}\int_{\mathbf{U}}\left|\mathrm{Re}\frac{e^{i\alpha
}(1-|\omega|^2)}{( \omega-z)(z \bar \omega -1)}\right|^p\,dA(\omega)
.$$ As in Lemma~\ref{film} for a fixed $z$, we introduce the change
of variables

\[\begin{aligned} \omega=\frac{z-a}{1-\bar za}.  \end{aligned} \]
Then
\begin{align*} \mathcal I_p&=2\int_{\mathbb{U}}\abs{\mathrm{Re}\,\frac{e^{i\alpha }(1-|\omega|^2)}{( \omega-z)(z
\bar \omega -1)}\,
}^p d\mu(\omega)\\[2ex]&=
{2}\int_{\mathbb{U}}\abs{\mathrm{Re}\,\frac{e^{i\alpha
}(1-|\omega|^2)}{a \cdot (1-\bar
z\omega)^2}}^p\,d\mu(\omega)\\[2ex]& ={2}\int_{\mathbb{U}}\abs{\mathrm{Re}\,\frac{e^{i\alpha
}(1-|\omega|^2)}{a \cdot (1-\bar
z\omega)^2}}^p\frac{(1-|z|^2)^2}{|1-\bar za|^4}\,d\mu(a)\\[2ex]& =
2\int_{\mathbb{U}} \abs{\mathrm{Re}\,\frac{e^{i\alpha }}{a \cdot
(1-\bar
z\omega)^2}}^p \frac{(1-|a|^2)^p(1-|z|^2)^{2+p}}{|1-\bar z a|^{4+2p}}\,d\mu(a).\\[2ex]&
 \end{align*}
Since
 \[\begin{aligned} 1-\bar z\omega&=1-\bar z \frac{z-a}{1-\bar z a}\\&
 =\frac{1-|z|^2}{1-\bar za},
   \end{aligned} \]
as in the proof of Lemma~\ref{film}, we obtain $$\mathcal I_p =
2(1-|z|^2)^{2-p}\int_{\mathbf{U}}
(1-|a|^2)^p\abs{\mathrm{Re}\,\frac{e^{i\alpha }(1-\overline z
a)^2}{a}}^p\frac{1}{|1-\bar za|^{4+2p}}\,d\mu(a).$$ Introducing
polar coordinates $a=r e^{i x}$ we have that
$$\mathcal I_p = \frac{2}{\pi}(1-\rho^2)^{2-p}\int_0^1\int_0^{2\pi}(1-r^2)^p\frac{\abs{\cos
x+r^2\rho^2\cos x-2 r \rho }^p}{(1+r^2\rho^2-2 r\rho \cos
x)^{2+p}}dx dr.$$ Let $\tau = \frac{2r\rho}{1+r^2\rho^2}$. It is
clear that $0\le \tau \le 1$. Then
\[\begin{split}\int_0^{2\pi}&\frac{\abs{\cos x+r^2\rho^2\cos x-2 r
\rho }^p}{(1+r^2\rho^2-2 r\rho \cos x)^{2+p}}dx \\&=
\frac{1}{(1+r^2\rho^2)^2} \int_{-\pi}^{\pi}\left|\frac{\tau -\cos
x}{1-\tau \cos x}\right|^p\frac{1}{(1-\tau \cos x)^2}dx.
\\&=\frac{2}{(1+r^2\rho^2)^2} \int_{0}^{\pi}\left|\frac{\tau -\cos
x}{1-\tau \cos x}\right|^p\frac{1}{(1-\tau \cos
x)^2}dx.\end{split}\] Introducing the change $$t=\frac{\tau -\cos
x}{1-\tau \cos x},$$ or what is the same $$\cos x = \frac{\tau -
t}{1-\tau t},$$ we obtain
\[\begin{split}\int_{0}^{\pi}\left|\frac{\tau -\cos x}{1-\tau \cos
x}\right|^p\frac{1}{(1-\tau \cos x)^2}dx&= \int_{-1}^1|t|^p
\frac{(1-\tau
t)(1-\tau^2)^{-1}}{(1-t^2)^{1/2}(1-\tau^2)^{1/2}}dt\\&=(1-\tau^2)^{-3/2}\int_{-1}^1
\frac{|t|^p}{(1-t^2)^{1/2}}dt\\&=(1-\tau^2)^{-3/2}\frac{\sqrt{\pi}\,
\Gamma[(1 + p)/2]}{ \Gamma[1 + p/2]}.\end{split}\] Therefore
\begin{equation}\label{math1}\mathcal I_p = \frac{4}{\sqrt \pi}
\frac{\Gamma[(1 + p)/2]}{\Gamma[1 +
p/2]}(1-\rho^2)^{2-p}\int_0^1r^{1-p}(1-r^2)^p\frac{1+r^2\rho^2}{(1-r^2\rho^2)^3}dr\end{equation}
i.e.
\begin{equation}\label{math}\mathcal I_p =\frac{2}{\sqrt \pi}
\frac{\Gamma[(1 + p)/2]}{\Gamma[1 + p/2]}I_p,\end{equation} where
$$I_p=2
(1-\rho^2)^{2-p}\int_0^1r^{1-p}(1-r^2)^p\frac{1+r^2\rho^2}{(1-r^2\rho^2)^3}dr.$$
Now \eqref{ess} follows from \eqref{essential}.

From \eqref{math} and \eqref{veryimportant} we obtain $$\mathcal
I_p(0)= \frac{ 2\Gamma[1+p]\Gamma[1 - \frac p2] \Gamma[\frac{1 + p}{
  2}]}{\sqrt{\pi}\Gamma[1 + \frac p2] \Gamma[2 + \frac p2]  },$$ and $$
\lim_{\rho\to 1}\mathcal I_p(\rho)=\frac{\Gamma[1+p]\Gamma[1-\frac
p2]\Gamma[\frac{1+p}{2}]\Gamma[3-p]}{\sqrt{\pi}\Gamma[2-\frac
p2]\Gamma[1+\frac p2]}.$$
\end{proof}

\begin{corollary}

For $z\in \mathbf{U}$ and $1\le p<2$ there hold the equalities
\begin{equation}\label{important1}
 \begin{split} &\int_{\mathbf{U}}\left|\mathrm{Re}\frac{(1-|\omega|^2)}{(
\omega-z)(z \bar \omega -1)}\right|^p\,dA(\omega) \\&=
\int_{\mathbf{U}}\left|\mathrm{Im}\frac{(1-|\omega|^2)}{(
\omega-z)(z \bar \omega -1)}\right|^p\,dA(\omega)
\\&=\frac{\sqrt{\pi}\Gamma[1 + p/2]}{2 \Gamma[(1 + p)/2]}\int_{\mathbf{U}}\left|\frac{(1-|\omega|^2)}{( \omega-z)(z \bar \omega
-1)}\right|^p\,dA(\omega).
\end{split}
\end{equation}

\end{corollary}
Notice that for $1\le p<2$ $$\frac{\pi}{4}\le
\frac{\sqrt{\pi}\Gamma[1 + p/2]}{2 \Gamma[(1 + p)/2]}<
\frac{\sqrt{\pi}\Gamma[1 + 2/2]}{2 \Gamma[(1 + 2)/2]} =1.$$

\section{$L^\infty$ norm of gradient}

\begin{theorem}\label{11}
If $u\in W_0^{2,p}$ is a solution, in the sense of distributions, of
Dirichlet's problem $ u_{z\overline z} = g(z)$, $g\in
L^q(\mathbf{U})$, $q>2$, $1/p+1/q=1$, then for
\begin{equation}\label{c}c^p_p=\mathrm{B}(1+p,1-p/2),\end{equation} where
$\mathrm{B}$ is the beta function, and
\begin{equation}\label{C}C^p_p=\frac{2^{2-p} \Gamma[(1 + p)/2]}{\sqrt{\pi}\Gamma[1
+ p/2]}c^p_p \end{equation} there hold the following sharp
inequalities
\begin{equation}\label{thirdi1}|\partial
u(z)|\le c_p\Vert g\Vert_q ,\ \ z\in \mathbf{U},\end{equation}
\begin{equation}\label{thirdi4}|\bar\partial
u(z)|\le c_p\Vert g\Vert_q ,\ \ z\in \mathbf{U},\end{equation}
 \begin{equation}\label{shtir} |\nabla u(z)|\le C_p\|g\|_{q},\ z\in
\mathbf{U}.
\end{equation}
The condition $q>2$ is the best possible. From
\eqref{thirdi1}--\eqref{shtir} we have the following relations
\begin{equation}\label{shkurt}\|\mathcal{C}_{\mathbf{U}}\|_{L^q\to L^\infty}
= c_p,\end{equation} \begin{equation}\label{mars}\|\mathcal{\bar
C}_{\mathbf{U}}\|_{L^q\to L^\infty} = c_p,\end{equation} and
\begin{equation}\label{prill}\|\mathcal{D}_{\mathbf{U}}\|_{L^q\to L^\infty}  =
C_p.\end{equation}

\end{theorem}
In the following example it is shown that the condition $q>2$ i.e.
$p<2$ in Theorem~\ref{11} is the best possible.
\begin{example}\label{she}
For $z\in \mathbf{U}\setminus\{0\}$ define $g$ by
$$g(z) =
\frac{z}{|z| \log|z|}\left(\frac{1-|z|^2}{|z|}\right).$$ It is easy
to verify that $g(z)\in L^2(\mathbf{U})$. On the other hand for the
solution $u\in W_0^{1,2}$ of Poisson equation $\Delta u = 4g$,
$\nabla u(0)$ do not exist and $\nabla u(z)$ is unbounded in every
neighborhood of $0$.

\end{example}
Since $C_1=\frac{16}{3\pi}$ and $c_1=\frac 43$ we obtain
\begin{corollary}
Under the condition of Theorem~\ref{11} for $p=1$, i.e. $q=\infty$
we have the following sharp inequalities $$\|\nabla u\|_{\infty}\le
\frac{16}{3\pi}\|g\|_\infty,$$

$$\|\partial u\|_{\infty}\le
\frac{4}{3}\|g\|_\infty,$$

$$\|\bar \partial u\|_{\infty}\le
\frac{4}{3}\|g\|_\infty.$$
\end{corollary}

For a positive nondecreasing continuous function $\omega:[0,l]\to
\mathbf R$, $\omega(0)=0$ we will say that is Dini's continuous if
it satisfies the condition
\begin{equation}\label{dini}\int_{0}^{l} \frac{\omega(t)}{t}
dt<\infty.\end{equation} A smooth Jordan curve $\gamma$ with the
length $l=|\gamma|$, is said to be Dini's smooth if the derivative
of its natural parametrization $g$ has the modulus of continuity
$\omega$ which is Dini's continuous.
\begin{proposition}[Kellogg]\label{conformalc}(See \cite{po} and \cite{wa})  {\it Let $\gamma$ be a Dini's smooth Jordan curve and let
 $\Omega=Int(\gamma)$. If $\varphi$ is a
conformal mapping of $\mathbf{U}$ onto $\varphi$, then $\varphi'$
and $\log \varphi'$ are continuous on $\mathbf{\overline{U}}$}.
\end{proposition}
For a conformal mapping $\varphi$ there holds
\begin{equation}\label{forma1}\Delta (u\circ \varphi)(z) = |\varphi'(z)|^2\Delta
u(\varphi(z)),\end{equation}  and
\begin{equation}\label{forma2}|\nabla (u\circ \varphi)(z)|
=|\varphi'(z)||\nabla u(\varphi(z))|.\end{equation} By using
Theorem~\ref{11} and Proposition~\ref{conformalc} and relations
\eqref{forma1} and \eqref{forma2}, we obtain
\begin{corollary}\label{rr}
Let $\Omega$ be a Jordan domain bounded by a Dini's smooth Jordan
curve  $\gamma$. If $u\in W_0^{2,p}(\Omega)$ is a solution, in the
sense of distributions, of Dirichlet's problem $u_{z\overline z} =
g(z)$, $g\in L^q(\Omega)$, $q>2$, $1/p+1/q=1$, then
\begin{equation}\label{1thirdi1}|\partial
u(z)|\le c_pC_\Omega\Vert g\Vert_q \ \ z\in \Omega,\end{equation}
\begin{equation}\label{1thirdi4}|\bar\partial
u(z)|\le c_pC_\Omega\Vert g\Vert_q ,\ \ z\in\Omega,\end{equation}
 \begin{equation}\label{1shtir} |\nabla u(z)|\le C_pC_\Omega\|g\|_{q},\ z\in
\Omega,
\end{equation}
with the constant $c_p$ and $C_p$ defined in \eqref{c} and \eqref{C}
 and $$C_\Omega=\inf\left\{\frac{\max\{|\varphi'(z)|^2:z\in\mathbf{T}\}}{\min\{|\varphi'(z)|:z\in\mathbf{T}\}}\right\},$$ where $\varphi$
ranges over all conformal mappings of the unit disk onto $\Omega$.

\end{corollary}

\begin{proof}[Proof of Theorem~\ref{11}]
We start by the formula \eqref{ca1} to obtain
$$|\partial u(z)|\le  \int_{\mathbf{U}}\frac{1-|\omega|^2}{|\omega-z |\cdot |\bar\omega z-1|} |g(\omega)|d\mu(\omega). $$  According to
H\"older's inequality it follows that
\[\begin{split}|\partial u(z)|&\le \left(\int_{\mathbf{U}} \left(\frac{1-|\omega|^2}{|z -
\omega|\cdot|1-\bar z\omega|} \right)^p d\mu(\omega)\right)^{1/p}
\left(\int_{\mathbf{U}}
|g(\omega)|^qd\mu(\omega)\right)^{1/q}.\end{split}\] By using
Lemma~\ref{film} we obtain $$|\nabla u(z)|^p\le I_p(0)\Vert
g\Vert_q^p=c_p^p\Vert g\Vert_q^p.$$ The inequality~\eqref{thirdi1}
easily follows. To show that the inequality is sharp take
\begin{equation}\label{gg}g(z) =
-\frac{z}{|z|}\left(\frac{1-|z|^2}{|z|}\right)^{p-1}.\end{equation}
Then
\[\begin{split}\partial u(0) &=
\frac{1}{\pi}\int_{B^2}\frac{1-|\omega|^2}{-w}\cdot(-\frac{\omega}{|\omega|}\left(\frac{1-|\omega|^2}{|w|}\right)^{p-1}
)dA(\omega)
\\&=\frac{1}{\pi}\int_{B^2}\left(\frac{1-|\omega|^2}{|\omega|}\right)^{p}
dA(\omega)\\&=\frac{1}{\pi}\int_0^1\int_0^{2\pi}r^{1-p}(1-r^2)^p
dr\,dt\\&=2\int_0^1r^{1-p}(1-r^2)^p
dr\\&=\mathrm{B}(1+p,1-p/2)\\&=I_p(0)\\&=
I_p(0)^{1/p}I_p(0)^{1/q}\\&= c_p \Vert g\Vert _q.\end{split}\] To
prove the inequality \eqref{shtir}, we begin by the equality
\begin{equation}\label{e}\begin{aligned}\nabla
u(z)h= {2} \int_{\mathbf{U}}\left<\frac{(1-|\omega|^2)}{(
\omega-z)(z \bar \omega -1)}, h\right>\, g(\omega)\,d\mu(\omega),\ \
h=e^{i \varphi}\in \mathbf{T}.\end{aligned}\end{equation} It follows
that $$|\nabla u(z)h|\le   {2}\left(
\int_{\mathbf{U}}\abs{\mathrm{Re}\frac{e^{-i
\varphi}(1-|\omega|^2)}{( \omega-z)(z \bar \omega
-1)}}^p\,d\mu(\omega)\right)^{1/p}\Vert g\Vert_q.$$ Lemma~\ref{le}
implies $$|\nabla u(z)|\le {2} \left(\frac{\Gamma[1+p]\Gamma[1 -
\frac p2] \Gamma[\frac{1 + p}{
  2}]}{{\sqrt{\pi}\Gamma[1 + \frac p2] \Gamma[2 + \frac p2] }}\right)^{1/p}\Vert g\Vert_q.$$ The equality is achieved
by  the following function

$$g(re^{it})=\left|\mathrm{Re}\frac{1-r^2}{re^{it}}\right|^{\frac
pq}\sign(\cos t).$$ Namely \[\begin{split}||g||^q_q
&=\frac{1}{\pi}\int_0^{2\pi}|\cos t|^p
dt\int_0^1r\left(\frac{1-r^2}{r}\right)^pdr\\&=\frac{2}{\sqrt
\pi}\frac{\Gamma((1+p)/2)}{\Gamma(1+p/2)}\int_0^1r^{1-p}(1-r^2)^{p}dr\\&=2^{p-2}C_p^p\end{split}\]
and

\[\begin{split}\nabla u(0)(1,0)& =\frac{1} {2\pi} \int_{\mathbf{U}}\mathrm{Re}\,\left(\frac{1-|\omega|^2}{ -\omega}\right)\,
g(\omega)\,dA(\omega)\\&= \frac{1}{2\pi}\int_0^{2\pi}|\cos
t|^{1+\frac pq}dt\int_0^1 r^{1-p}(1-r^2)^pdr\\&=\frac{1}{\sqrt
\pi}\frac{\Gamma((1+p)/2)}{\Gamma(1+p/2)}\int_0^1r^{1-p}(1-r^2)^{p}dr\\&=2^{p-3}C_p^p=C_p\|g\|_{q}.\end{split}\]



\end{proof}

\begin{remark}
The solution to the  Poisson's equation with homogeneous boundary
condition over the unit disk with $g$ defined in \eqref{gg} for
$p=1$, i.e. for $g(z)=-z/|z|$ is
$$u(z) = \frac{4}{3}e^{-it}(r - r^2)=\frac{4}{3}\bar z(1-|z|) .$$  It would be of interest to find the solution for
arbitrary $1\le p<2$.

\end{remark}


\section{$L^p$ norm of Cauchy transform}

In this section we consider the situation $1\le p<\infty$.

\begin{theorem}\label{ok}
If $u$ is a solution, in the sense of distributions,  of Dirichlet's
problem $ u_{z\overline z} = g(z)$, $u\in W^{1,p}_0(\mathbf{U})$,
$z\in \mathbf{U}$, $g\in L^p(\mathbf{U})$, $1< p< \infty $, then

\begin{equation}\label{third}\Vert \nabla
u\Vert _p\le 4\left(\frac{4}{3\pi}\right)^{1-1/p}\Vert g\Vert
_p\end{equation}
\begin{equation}\label{thirdsec}\Vert \partial u\Vert _p\le
2\left(\frac{2}{3}\right)^{1-1/p}\Vert g\Vert _p\end{equation} and
\begin{equation}\label{thi}\Vert \bar\partial u\Vert _p\le
2\left(\frac{2}{3}\right)^{1-1/p}\Vert g\Vert _p\end{equation}

 \end{theorem}
\begin{remark}

The inequalities \eqref{third} and \eqref{thirdsec} are
asymptotically sharp as $p$ approaches $1$ or $\infty$. We will
treat the Hilbert case $p=2$ separately in order to obtain the sharp
constant for the case $p=2$ (see Section~5). Making use of this fact
and by using Riesz-Thorin interpolation theorem, we will improve
inequalities \eqref{third} and \eqref{thirdsec} (see Section 6). It
remains  to find sharp inequalities for $0<p<2$ and $2<p<\infty$.
\end{remark}
\begin{corollary}\label{pasoja} Under the conditions of Theorem~\ref{ok} there hold the
following sharp inequalities
\begin{equation}\label{th}\Vert \nabla
u\Vert_1 \le 4\Vert g\Vert_1, \end{equation}
\begin{equation}\label{th1}\Vert \partial
u\Vert_1 \le 2\Vert g\Vert_1,\end{equation} and
\begin{equation}\label{th2}\Vert \bar\partial
u\Vert_1 \le 2\Vert g\Vert_1 .\end{equation}
\end{corollary}

\begin{proof}[Proof of Corollary~\ref{pasoja}]
Let $u_n$ be a solution of $\Delta u = 4g_n$ for $g_n =
n^2\chi_{\frac 1n \mathbf{U}}$. By using \eqref{ca1} and polar
coordinates, we obtain

\[\begin{split}\frac{\partial u_n}{\partial z} &= \frac{1}{\pi}\int_{\mathbf{U}}
\frac{1-|\omega|^2}{(\omega-z )(\bar\omega z-1)}
g_n(\omega)dA(\omega)\\&=\frac{n^2}{\pi}\int_{1/n\mathbf{U}}
\frac{1-|\omega|^2}{(\omega-z )(\bar\omega z-1)}
dA(\omega)\\&=\frac{n^2}{\pi}\int_0^{1/n}r(1-r^2)dr\int_0^{2\pi}
\frac{1}{(r e^{it}-z )(re^{-it}
z-1)}dt\\&=\frac{n^2}{\pi}\int_0^{1/n}r\int_{|\zeta|=1}\frac{(1-r^2)d\zeta}{i(r
\zeta-z )(r z-\zeta)}\end{split}\]
Let
$$\lambda_z(r)=\int_{|\zeta|=1}\frac{d\zeta}{i(r \zeta-z )(r
z-\zeta)}.$$ Then by Cauchy residue theorem, for almost every $r$
$$\displaystyle\lambda_z(r)= \mathrm{Ind}_{\mathbf T}\left(\frac
zr\right)\Res_{\zeta=\frac zr}\frac{2\pi i(1-r^2)}{i(r \zeta-z )(r
z-\zeta)}+\Res_{\zeta=zr}\frac{2\pi i(1-r^2)}{i(r \zeta-z )(r
z-\zeta)}.$$ Therefore  $$\lambda_z(r)=\left\{
                            \begin{array}{ll}
                              0, & \hbox{if $\abs{\frac zr}<1$;} \\
                              \frac{2\pi}{z}, & \hbox{if $\abs{\frac zr}>1$.}
                            \end{array}
                          \right.$$
Thus
$$\displaystyle\frac{\partial u_n}{\partial z} =\left\{
                            \begin{array}{ll}
                             2n^2 \int_0^{|z|}\frac{r}{z}dr, & \hbox{if $\abs{ z}<\frac 1n$;} \\
                              2n^2 \int_0^{1/n}\frac{r}{z}dr, & \hbox{if $\frac 1n\le \abs{z}\le 1$.}
                            \end{array}
                         \right.$$
It follows that
$$\frac{\partial u_n}{\partial z} =\left\{
                                              \begin{array}{ll}
                                                 \frac{|z|^2 n^2}{z}, & \hbox{if $|z|<1/n$;} \\
                                                \frac{1 }{z}, & \hbox{if $1/n\le |z|<1$.}
                                              \end{array}
                                            \right.$$
Then
$$a_n:=\int_{\mathbf{U}} |\frac{\partial u_n}{\partial z}|d\mu
=2\left(1-\frac 2{3n}\right).$$ On the other hand
$$b_n:=\int_{\mathbf{U}}|g_n(z)|d\mu(z)=1.$$ Since $$\lim_{n\to
\infty} \frac{a_n}{b_n}=2,$$ the inequality \eqref{th1} is sharp. On
the other hand since $u_n$ is a real function, it follows that
$$|\nabla u_n(z)|=2\left|\frac{\partial u_n}{\partial z}\right|$$ and this shows that
\eqref{th} is sharp.

Observe that the sequence $g_n$ converges to Dirac delta function
$$\delta(0)=\infty, \delta(z)=0, z\neq 0, \int_{\mathbf{U}} \delta(z)
d\mu = 1.$$ The solution $u_n\in W^{2,p}_0$ of equation $\Delta u =
g_n$ converges to $u(z)=\frac 1z$ which is a solution of $\Delta u =
\delta$.
\end{proof}
In order to prove Theorem~\ref{ok}, we need the following lemmas.
\begin{lemma}[Young's inequality for convolution]\label{le1}\cite[pp. 54-55; 8, Theorem
20.18]{young} If $h\in L^p(\Bbb R)$, and $\rho\in C^\infty_c(\Bbb
R)$, such that $\int_{\Bbb R}\rho(t) dt=1$, then $$\Vert h*\rho\Vert
_p\le \Vert h\Vert _p.$$

\end{lemma}

 We make use of the following
immediate corollary of \cite[Lemma~9.17]{gt}:
\begin{lemma}[Stability Lemma] Let $v$ be a weak solution
of ƒ$\Delta v = h$ with $v \in W^{1,p}_0$, $1<p<\infty$. Then for
some $C=C(p)$
\begin{equation}\label{la}\Vert \nabla v\Vert _{L^{p}}\le
C\Vert h\Vert _p.\end{equation}
\end{lemma}


\begin{lemma}
For  $\omega\in \mathbf{U}$ the function defined by
\begin{equation}\label{j}J(\omega):=\frac 12\int_{\mathbf{U}}
\frac{1-|\omega|^2}{|z - \omega|\cdot|1-\bar z\omega|}
d\mu(z),\end{equation} is equal to
\begin{equation}\label{enj}J(\omega) = \frac{1 - |\omega|^2}{2|\omega|}
\log\frac{1+|\omega|}{1-|\omega|}.\end{equation}
\end{lemma}

\begin{proof}

 For a fixed $\omega$, we introduce the
change of variables $$a(z) = \frac{\omega-z}{1-z\overline \omega},$$
and recall that $$a'(z) = -\frac{1-|\omega|^2}{(1-z\overline
\omega)^2}.$$ We obtain
\[\begin{split}J(\omega)&=\frac 12\int_{\mathbf{U}}
\frac{1-|\omega|^2}{|1-z\overline \omega|^2}
\frac{1}{|a(z)||a'(z)|^2} d\mu(a)\\&=\frac 12\int_{\mathbf{U}}
\frac{|1-a\overline \omega|^{-2}}{(1-|\omega|^2)^{-1}}|a|^{-1}
d\mu(a)
\\&=\frac 12(1-|\omega|^2)\int_{\mathbf{U}} |a|^{-1}|1-a\overline
\omega|^{-2}d\mu(a)\\&= \frac{1}{2\pi}(1-|\omega|^2)\int_{0}^1dt
\int_0^{2\pi}|(1-\overline \omega re^{it})^{-1}|^{2}dt.\end{split}\]
By using Parserval's formula to the function $$f(a) = (1-\overline
\omega a)^{-1} = \sum_{k=0}^\infty  \bar \omega^k a^k ,
$$ we obtain
\[\begin{split}J(\omega) &= (1-|\omega|^2)\sum_{k=0}^\infty
\frac{|\omega|^{2k}}{{2k+1}}\\&=\frac{1 - |\omega|^2}{2|\omega|}
\log\frac{1+|\omega|}{1-|\omega|}.
\end{split}\]

\end{proof}

\begin{proof}[Proof of Theorem~\ref{ok}]

Let $\mathcal I_1(z)$ be the function defined by \eqref{graso}. We
introduce appropriate mollifiers: Fix a smooth function
$\rho:\R\to[0,1]$ which is compactly supported in the interval
$(-1,1)$ and satisfies $\int_\R\rho=1$. For $\varepsilon>0$ consider
the mollifier
\begin{equation}\label{mol}
     \rho_\varepsilon(t):=\frac{1}{\varepsilon}\,
     \rho\left(\frac{t}{\varepsilon}\right).
\end{equation}
It is compactly supported in the interval
$(-\varepsilon,\varepsilon)$ and satisfies
$\int_\R\rho_\varepsilon=1$. For $\varepsilon>0$ define
$$g_\varepsilon(x) = \int_{\Bbb R} g(y)
\frac{1}{\varepsilon}\rho(\frac{x-y}{\varepsilon})dy=\int_{\Bbb R}
g(x-\varepsilon z)\rho(z)dz.$$ Then $g_\varepsilon$ converges to $g$
as $\varepsilon\to 0$ in $L^p$ norm. Let $u_\varepsilon(z)\in
C_0^\infty(\mathbf{U})$ be a solution to $ u_{z\bar z} =
g_\varepsilon$. Then by using the Jensen's inequality, and having in
mind the fact that the measure
$$d\nu_z(\omega):=2\abs{\mathrm{Re}\frac{e^{-i\varphi}(1-|\omega|^2)}{(\omega-z )(\bar\omega z-1)}}
\frac{d\mu(\omega)}{\mathcal I_1(z)}$$ is a probability measure in
the unit disk, for $h=e^{i\varphi}$, by using \eqref{e:POISSON1}, we
obtain

\[\begin{split}|\nabla u_\varepsilon(z)h|^p&=\left|\int_{\mathbf{U}}2\left<\frac{1-|\omega|^2}{(\omega-z )(\bar\omega z-1)}
,h\right>g_\varepsilon(\omega)d\mu(\omega)\right|^p
\\&\le  \left(\int_{\mathbf{U}}2\abs{\mathrm{Re}\frac{e^{-i\varphi}(1-|\omega|^2)}{(\omega-z )(\bar\omega z-1)}}\abs{g_\varepsilon(\omega)}
d\mu(\omega)\right)^p\\&=\mathcal
I_1^{p}(z)\left(\int_{\mathbf{U}}|g_\varepsilon(\omega)|\abs{\mathrm{Re}\frac{2e^{-i\varphi}(1-|\omega|^2)}{(\omega-z
)(\bar\omega z-1)}} \frac{d\mu(\omega)}{\mathcal I_1(z)}\right)^p
\\&\le \mathcal I_1^{p-1}(z)\int_{\mathbf{U}} \abs{\mathrm{Re}\frac{2e^{-i\varphi}(1-|\omega|^2)}{(\omega-z )(\bar\omega z-1)}}
|g_\varepsilon(\omega)|^p d\mu(\omega).
\end{split}\]
Thus \begin{equation}\label{inte} |\nabla u_\varepsilon(z)|^p\le
\mathcal I_1^{p-1}(z)\int_{\mathbf{U}}
\abs{\frac{2(1-|\omega|^2)}{(\omega-z )(\bar\omega z-1)}}
|g_\varepsilon(\omega)|^p d\mu(\omega).
\end{equation}
Integrating \eqref{inte} over the unit disk $\mathbf U$, by using
Fubini's theorem it follows that
\[\begin{split}\int_{\mathbf{U}}|\nabla u_\varepsilon(z)|^pd\mu(z)&\le \int_{\mathbf{U}}\mathcal I_1^{p-1}(z)\int_{\mathbf{U}}
 \abs{\frac{2(1-|\omega|^2)}{(\omega-z )(\bar\omega z-1)}} |g_\varepsilon(\omega)|^p
d\mu(\omega)d\mu(z)\\&=\int_{\mathbf{U}} |g_\varepsilon(\omega)|^p
d\mu(\omega)\int_{\mathbf{U}}\mathcal I_1^{p-1}(z)
\abs{\frac{2(1-|\omega|^2)}{(\omega-z )(\bar\omega z-1)}} d\mu(z).
\end{split}\]
Since \begin{equation*}\mathcal I_1=\frac{16 {_2F_1}(-1/2; 1; 5/2;
|z|^2)}{3\pi}\le \frac{16}{3\pi},\end{equation*} we obtain first
that
$$\mathcal I_1^{p-1}\le  \left(\frac{16}{3\pi}\right)^{p-1}.$$ On the other hand by \eqref{enj},
$$\int_{\mathbf{U}} \frac{2(1-|\omega|^2)}{|z - \omega|\cdot|1-\bar
z\omega|}d\mu(z)\le 4.$$ It follows that

$$\Vert \nabla u_\varepsilon\Vert^p _p\le
4\left(\frac{16}{3\pi}\right)^{p-1}\Vert g_\varepsilon\Vert^p _p,$$
i.e.
$$\Vert \nabla u_\varepsilon\Vert _p\le
4\left(\frac{4}{3\pi}\right)^{1-1/p}\Vert g_\varepsilon\Vert _p.$$
Take $h=g-g_\varepsilon$. Then by \eqref{la} we obtain
$$\Vert \nabla u_\varepsilon- \nabla u\Vert _p\le C\Vert g-g_\varepsilon\Vert _p.$$
This implies that  $$\lim_{\varepsilon\to 0}\Vert \nabla
u_\varepsilon- \nabla u\Vert _p=\lim_{\varepsilon\to 0} C\Vert
g-g_\varepsilon\Vert _p=0$$ and therefore
\begin{equation}\label{slo} \lim_{\varepsilon\to 0}\Vert \nabla
u_\varepsilon\Vert _p^p= \Vert \nabla u\Vert _p^p.\end{equation} On
the other hand by Lemma~\ref{le1} we obtain
\begin{equation}\label{cro}\Vert g_\varepsilon\Vert _p^p\le \Vert g\Vert _p^p.\end{equation}
The conclusion is $$\Vert \nabla u\Vert _p\le
4\left(\frac{4}{3\pi}\right)^{1/p}\Vert g\Vert _p,$$ as desired. An
analogous proof yields the inequalities \eqref{thirdsec} and
\eqref{thi}. In this case we make use of the following probability
measure
$$d\mu_z(\omega):=\abs{\frac{e^{-i\varphi}(1-|\omega|^2)}{(\omega-z
)(\bar\omega z-1)}} \frac{d\mu(\omega)}{I_1(z)},$$ and the relation
\begin{equation*} I_1=\frac{4 {_2F_1}(-1/2; 1; 5/2; |z|^2)}{3}\le
\frac{4}{3},\end{equation*} which coincides with \eqref{petar}.

\end{proof}
\section{The Hilbert norm of Cauchy transform}


In this section we determine the precise value of the operator norm
$\mathcal C_{\mathbf U}$ when  considered as an operator from the
Hilbert space $L^2(\mathbf U)$ into itself. It follows from our
proof that the Hilbert norm of $C_{\mathbf U}$ coincides with the
Hilbert norm of $\mathfrak C:L^2(\mathbf U)\to L^2(\mathbf U)$,
which has been determined by Anderson and Hinkkanen in \cite{ah}.
For $k\in \mathbf Z$, we denote by $\alpha_k$ the smallest positive
zero of the Bessel function
$$J_k(x) = \sum_{m=0}^\infty \frac{(-1)^m}{m! \, (m+k)!}
{\left({\frac{x}{2}}\right)}^{2m+k},$$ the smallest positive zero of
the Bessel function $J_0$
 satisfies $\alpha:=\alpha_0\approx 2.4048256$ by
\cite[p.~748]{watson}, and hence $\frac {2}{\alpha}\approx 0.83166$.

Concerning the zeros of Bessel functions there hold the following
\begin{lemma}\cite[section~15.22,~p.~479]{watson}\label{air}
The sequence $\alpha_k$ is increasing for $k\ge 0$.
\end{lemma}

The main theorem of this section is the theorem:

\begin{theorem}\label{teorema}
The norm of the operator $C_{\mathbf U}:L^2(\mathbf U)\to
L^2(\mathbf U)$ is
$$\|C_{\mathbf U}\|= \frac{2}{\alpha}.$$ In other words
\begin{equation}\label{hilbert}\|\mathcal C_{\mathbf U} [g]\|_2\le \frac{2}{\alpha}\|g\|_2,\
\ \text{for}, \ \ g\in L^2(\Bbb U).\end{equation} The equality holds
in \eqref{hilbert} if and only if $g(z)=c|z|J_0(\alpha |z|)$, for
a.e. $z\in \mathbf U$, where $c$ is a complex constant.
\end{theorem}

To prove Theorem~\ref{teorema} it suffices to show that
\begin{equation}\label{sufa}\|\mathcal C_{\mathbf U}[P]\|\le \frac{2}{\alpha}\|P\|_2\end{equation} whenever
\begin{equation}\label{nice} P(z) = \sum_{n=0}^\infty\sum_{m=0}^\infty a_{mn}z^m\bar z^n\end{equation} is a polynomial
in $z$ and $\bar z$, since such functions are dense in $L^2(\mathbf
U)$ and $\frac{2}{\alpha}$ is the best constant. In this case only
finitely many of the complex numbers $a_{mn}$ are nonzero. It is
evident that there exist radial functions $f_d$, $d\in \mathbf Z$
such that
\begin{equation}\label{fd} P(z)=\sum_{d=-\infty}^{\infty} g_d(z) ,
\end{equation}
where $g_d(z)=f_d(r) e^{i d t}$, $d=m-n$. Observe that $g_{d_1}$ and
$g_{d_2}$ are orthogonal for $d_1\neq d_2$ in Hilbert space
$L^2(\mathbf U)$. In the following proof we will show that $\mathcal
C_{\mathbf U}[g_{d_1}]$ and $\mathcal C_{\mathbf U}[g_{d_2}]$ are
orthogonal in Hilbert space $L^2(\mathbf U)$.

Thus $$\|\mathcal C_{\mathbf U}[P]\|_2\le \frac{2}{\alpha}\|P\|_2$$
if and only if

$$\sum_{d=-\infty}^\infty\|\mathcal C_{\mathbf U}[g_d]\|_2\le  \frac{2}{\alpha}\sum_{d=-\infty}^\infty\|g_d\|_2.$$ We
will show a bit more, we will prove the following lemma, which is
the main ingredient of the proof of Theorem~\ref{teorema}.
\begin{lemma}\label{malet}
For $d\in\mathbf Z$ there holds the following sharp inequality
\begin{equation}\label{shala}\|\mathcal C_{\mathbf U}[g_d]\|_2\le
\frac{2}{\alpha_{|d|}}\|g_d\|_2,\end{equation} where $g_d(z)=f_d(r)
e^{i d t}\in L^2(\mathbf U)$, $z=re^{it}$.
\end{lemma}

Before proving Lemma~\ref{malet} we need some preparation.

By using \eqref{ca1} and polar coordinates, we obtain
\[\begin{split}\mathcal C_{\mathbf U}[g_d] &= \frac{1}{\pi}\int_{\mathbf{U}}
\frac{1-|\omega|^2}{(\omega-z )(\bar\omega z-1)}
g_d(\omega)dA(\omega)\\&=\frac{1}{\pi}\int_{\mathbf{U}}
\frac{(1-|\omega|^2)f_d(r) e^{i d t}}{(\omega-z )(\bar\omega z-1)}
dA(\omega)\\&=\frac{1}{\pi}\int_0^{1}rf_d(r)dr\int_0^{2\pi}
\frac{(1-r^2)e^{idt}}{(r e^{it}-z )(re^{-it}
z-1)}dt\\&=\frac{1}{\pi}\int_0^{1}rf_d(r)\int_{|\zeta|=1}\frac{(1-r^2)\zeta^dd\zeta}{i(r
\zeta-z )(r z-\zeta)}.\end{split}\] Let
$$\lambda_z(r)=\int_{|\zeta|=1}\frac{\zeta^d (1-r^2)d\zeta}{i(r \zeta-z )(r
z-\zeta)}.$$ Then by Cauchy residue theorem, for every $r\neq |z|$
\[\begin{split}\displaystyle\lambda_z(r)&= \mathrm{Ind}_{\mathbf
T}\left(\frac zr\right)\Res_{\zeta=\frac zr}\frac{2\pi
i(1-r^2)\zeta^d}{i(r \zeta-z )(r z-\zeta)}\\&\ \ \ +
\Res_{\zeta=zr}\frac{2\pi i(1-r^2)\zeta^d}{i(r \zeta-z )(r
z-\zeta)}+\Res_{\zeta=0}\frac{2\pi i(1-r^2)\zeta^d}{i(r \zeta-z )(r
z-\zeta)}.\end{split}\] Thus
$$\displaystyle\lambda_z(r)=\left\{
                             \begin{array}{ll}
                               -\mathrm{Ind}_{\mathbf T}\left(\frac zr\right)\frac{2\pi z^{d-1}}{
r^{d}}+\frac{2\pi z^{d-1}}{ r^{d}}, & \hbox{if $d\le 0$;} \\
                      -\mathrm{Ind}_{\mathbf T}\left(\frac zr\right)\frac{2\pi z^{d-1}}{
r^{d}}+{2\pi z^{d-1}}{ r^{d}}, & \hbox{if $d>0$.}
                             \end{array}
                           \right.
$$ Therefore for $d\le 0$
$$\lambda_z(r)=\left\{
                            \begin{array}{ll}
                              0, & \hbox{if $r>|z|$;} \\
                              \frac{2\pi z^{d-1}}{ r^{d}}, & \hbox{if $r<|z|$,}
                            \end{array}
                          \right.$$
and hence \begin{equation}\label{dn}\mathcal C_{\mathbf U}[g_d]
=2{z^{d-1}}\int_0^{|z|}f_d(r)r^{-d+1}dr.\end{equation} For $d>0$
$$\lambda_z(r)=\left\{
                            \begin{array}{ll}
                              -\left(\frac zr\right)\frac{2\pi z^{d-1}}{
r^{d}}+{2\pi z^{d-1}}{ r^{d}}, & \hbox{if $r>|z|$} \\
                              {2\pi z^{d-1}}{ r^{d}}, & \hbox{if $r<|z|$,}
                            \end{array}
                          \right.$$
and therefore

\begin{equation}\label{dp}\mathcal C_{\mathbf U}[g_d]
=2{z^{d-1}}\left(\int_0^{1}f_d(r)r^{d+1}dr-\int_{|z|}^{1}f_d(r)r^{-d+1}dr\right).\end{equation}

First of all, it follows from \eqref{dn} and \eqref{dp} that for
$d_1\neq d_2$, $\mathcal C_{\mathbf U}[g_{d_1}]$ and $\mathcal
C_{\mathbf U}[g_{d_2}]$ are orthogonal.

In particular if $f_d(r) = r^{n+m}$, i.e. $g_d(z) = z^m \overline
z^n$, and $d=m-n\le 0$, then from \eqref{dn} we have that
\begin{equation}\label{posi} \mathcal C_{\mathbf U}[z^m \overline z^n] =
\frac{1}{n+1} z^m\bar z^{n+1}.
\end{equation}
If $d=m-n> 0$, from \eqref{dp} we obtain
\begin{equation}\label{nega} \mathcal C_{\mathbf U}[z^m \overline z^n] =
\frac{1}{n+1} \left(z^m\bar z^{n+1}-\frac{m-n}{m+1}z^{m-n-1}\right).
\end{equation}

\begin{proof}[Proof of Lemma~\ref{malet}]
We divide the proof into two cases.
\subsection{ Case $d=m-n\le 0$}
From \eqref{dn} we deduce that
$$L_d:=\int_{\mathbf U}\abs{\mathcal C_{\mathbf U}[g_d]}^2d\mu(z)=8\int_{0}^1r{r^{2d-2}}\left|\int_0^{r}f_d(s)s^{-d+1}ds\right|^2dr.$$
On the other hand let $$R_d := \int_0^1 r |f_d(r)|^2 dr.$$ We should
find the best constant $A_d$ such that
$$  L_d\le A_d R_d,$$ i.e. the best constant $B_d=\frac{A_d}{4}$
such that
$$\int_{0}^1r{r^{2d-2}}\left|\int_0^{r}f_d(s)s^{-d+1}ds\right|^2dr \le
B_d \int_0^1 r |f_d(r)|^2 dr.$$ Without loos of generality, assume
that $f_d$ is real and positive in $[0,1]$. By setting
$h(s)=f_d(s)s^{-d+1}$ we obtain the following inequality
$$\int_0^1{x^{2d-1}}\left(\int_0^rh(s) ds\right)^2dx \le B_d\int_0^1
{r^{2d-1}}(h(r))^2dr.$$

This problem, which involves an inequality of Hardy type, can be
solved by appealing to a more general result of Boyd (\cite[Theorem~
1,~p.~368]{boyd}) (Proposition~\ref{boy}).

To formulate the result of Boyd we need some definitions and facts.
For $\omega, m\in C^1(a, b)$, we assume that $w(x) > 0$ and $m(x)
> 0$ for $a < x < b$. By $T_1$ is defined the operator
$$T_1f(x)=\omega(x)^{1/p}m(x)^{-1/p}\int_a^xf(t)dx.$$
Let  $$L_m^r=\left\{f: \|f\|_s:=\left(\int_0^1|f|^s m(x)
dx\right)^{1/s}<\infty\right\}.$$ A simple sufficient condition for
$T$ to be compact from $L_m^r\rightarrow L_m^{s'}$ (here $s'$ is
conjugate of $s$ ($s'=s/(s-1)$)),  is that the function $k$ defined
by
$$k(x,t):=\omega(x)^{1/p}m(x)^{-1/p}m(t)^{-1}\chi_{[a,x]}(t)$$ have finite $(r',
s)$-double norm \begin{equation}\label{condition}
\|T\|=\left\{\int_a^b\left[\int_a^b
k(x,t)^{r'}m(t)dt\right]^{\frac{s}{r'}}m(x)dx\right\}^{\frac
1s}<\infty,\ \ s>1, r>1.
\end{equation}
Here $r'=r/(r-1)$. For this argument we refer to \cite[p.~319]{za}.

\begin{proposition}\cite{boyd}\label{boy} Suppose that $\omega, m\in C^1(a, b)$, that $w(x) > 0$ and $m(x)
> 0$ for $a < x < b$, that $p > 0$, $r > 1$, $0 \le q < r$, and that the
operator $T_1$ is compact from $L_m^r \rightarrow L_m^s (s = pr/(r -
q))$.

Then, the following eigenvalue problem (P) has solutions $(y,
\lambda)$, $y \in C^2(a, b)$ with $y(x)>0$, $y'(x)>0$ in $(a,b)$.

\begin{equation}\label{p}\left\{
  \begin{array}{ll}
     \frac{d}{dx}(r\lambda {y'}^{r-1}m-qy^p{y'}^{q-1}\omega)+py^{p-1}{y'}^q\omega=0 \\
     \lim_{x\to a+0}y(x)=0,\ \  \lim_{x\to b-0}(r\lambda {y'}^{r-1}m-qy^p{y'}^{q-1}\omega(x))=0\\
     \|y'\|_r=1.
  \end{array}
\right.
\end{equation}

There is a largest value $\lambda$ such that \eqref{p} has a
solution and if $\lambda^*$ denotes this value, then for any $f\in
L^r_m$,

\begin{equation}\label{5}\int_a^b\left|{\int_a^x f}\right|^p\abs{f}^q w(x) dx \le
\frac{r\lambda^*}{p+q}\left\{\int_a^b|f|^r m(x)
dx\right\}^{(p+q)/r}.\end{equation}

Equality holds in \eqref{5} if and only if $f=cy'$ a.e. where $y$ is
a solution of \eqref{p} corresponding to $\lambda=\lambda^*$, and
$c$ is any constant.
\end{proposition}

In our case we have $r=2$, $p=2$, $q=0$, $m(x)={x^{2d-1}}$ and
$\omega(x)={x^{2d-1}}$, $s=pr/(r-q)=2$ and $s'=s/(s-1)=2$. The
corresponding differential equality is equivalent to
\begin{equation}\label{quo}x^2y''+(2d-1) x y' +\frac{x^2}{\lambda} y
=0.\end{equation} The additional compactness condition required is
proved by proving that the operator
$$T_1 f(x)=\int_0^xf(t) dt$$ ic compact from $L^2_m\rightarrow L^2_m$. Namely by \eqref{condition} we have to
show that $\|T_1\|<\infty$. In this case
$$k(x,t)=t^{1-2d}\chi_{[0,x]}(t).$$ Thus
$$\|T_1\|=\left\{\int_0^1\int_0^x t^{2-4d} t^{2d -1} {\mathrm dt} x^{2d-1}
 {\mathrm dx}\right\}^{1/2}=\frac{1}{2\sqrt{1-d}}<\infty.$$
\\
The positive solution of \eqref{quo} is
\begin{equation}\label{possol}y(x) =
x^{-d+1}J_{-d+1}(\frac{x}{\sqrt{\lambda}}),\end{equation} where
$J_{-d+1}$ is the Bessel function.
\\
Then by \cite[Section~3.2,~p.~45]{watson} $$y'(x)
=\frac{x^{-d+1}}{\sqrt{\lambda}}J_{-d}(\frac{x}{\sqrt{\lambda}})$$
and therefore
$$y'(1)
=\frac{1}{\sqrt{\lambda}}J_{-d}(\frac{1}{\sqrt{\lambda}}).$$ 
Thus the largest possible value of $\lambda$ is $B_d=\lambda_{d}$
where
$$\frac{1}{\sqrt{\lambda_{d}}}=
\alpha_{-d},$$  and $\alpha_d$ is the smallest positive zero of
$J_{-d}$.
We note that then also $y(x) > 0$ for $0 < x < 1$ since
the smallest positive zero of $J_{-d+1}$ is larger than that of
$J_{-d}$. Finally
\begin{equation}\label{AA}A_d=\frac{4}{\alpha^2_{|d|}}\end{equation} as desired.

\subsection{The case $d>0$} Let $b_n = a_{n+d,n}$, where $a_{mn}$ are the coefficients of expresion \eqref{nice}.
Then $$g_d(z) =\sum_{n=0}^{\infty}b_{n}z^{n+d}\bar z^n.$$ Thus
\[\begin{split}\|g_d\|_2^2&=\frac{1}{\pi}\int_0^1\int_0^{2\pi}\left|\sum_{n=0}^\infty b_{n}r^{n+d}e^{i dt}\right|^2 r dr dt\\&
=2\sum_{n,l\ge 0}\frac{b_{n}\overline{b_{l}}}{n+l+2 d
+2}.\end{split}
\]
On the other hand, by using \eqref{nega}, we obtain

\[\begin{split}\|\mathcal C_{\mathbf U}[g_d]\|_2^2&=\frac{1}{\pi}\int_0^1\int_0^{2\pi}\left|\sum_{n\ge
0}\frac{b_{n}}{n+1}\left[w^{n+d} \bar w^{n+1}
-\left(\frac{d}{m+1}\right) w^{d-1}\right]\right|^2 r dr dt\\&
=\sum_{n,l\ge 0}
\frac{b_{n}\overline{b_{l}}}{(n+1)(l+1)}\left[\frac{1}{n+l+d+2}-\frac{d}{(l+d+1)(n+d+1)}\right]\\&=\sum_{n,l\ge
0}\frac{b_{n}\overline{b_{l}}}{(1 + l + d) (1 + n + d) (2 + l + n +
d)}.\end{split}\]

So we  seek the best constant $A_d$ such that for all choices of
$b_{n}\in  \mathbf{C}$ (since $P$ is a polinom, only finitely many
$b_{n}$ are nonzero, but the proof works as well, without this
assumption), there holds the inequality

\begin{equation}\label{quad}\sum_{l,n\ge 0}\frac{b_{n}\overline{b_{l}}}{(1 + l + d) (1 +
n + d) (2 + l + n + d)}\le A_d\sum_{l,n\ge
0}\frac{b_{n}\overline{b_{l}}}{n+l+d+1}.\end{equation}

Let  $$\psi(t)= \sum_{n=0}^\infty b_n t^n.$$

Then the previous inequality is equivalent with

\begin{equation}\label{qu}\int_0^1\frac{1}{r^{d+1}}\left|\int_0^r \psi(s) s^d ds\right|^2
dr\le A_d\int_0^1 r^d \abs{\psi(r)}^2dr.\end{equation}

Since the quadratic forms in \eqref{quad} are symmetric with real
coefficients ($b_n \overline{b_l}+b_l \overline{b_n}$ is a real
number), it suffices to consider the inequality \eqref{qu} for
arbitrary real-valued continuous functions on $[0, 1]$.  Setting
$h(r) = \psi(r) s^d$, the inequality is equivalent to
$$\int_0^1\frac{1}{r^{d+1}}\left(\int_0^r h(s) ds\right)^2 dr\le
A_d\int_0^1 \frac{1}{r^d} h^2(r)dr.$$

Here we again make use of Proposition~\ref{boy}. In this case we
have $r=2$, $p=2$, $q=0$, $s=s'=2$, $m(x)=\frac{1}{x^{d}}$ and
$\omega(x)=\frac{1}{x^{d+1}}$.
 The additional compactness
condition required is easily proved by observing that
$$T_1 f(x)=\frac 1{\sqrt x}\int_0^xf(t) dt$$ and applying
\eqref{condition} to show that it is compact from $L^2_m\rightarrow
L^2_m$. Namely
$$k(x,t)=x^{-1/2} t^d \chi_{[0,x]}(t),$$ and therefore
$$\|T_1\|=\left(\int_0^1\int_0^x\frac{t^{2d}}{x}\frac{{\mathrm d}t}{t^d}
\frac{{\mathrm d}x}{x^d}\right)^{1/2}=(d+1)^{-1/2}<\infty.$$

  The corresponding differential equality
is equivalent to
$$xy''-d y' +\frac{x}{\lambda} y =0,$$ which can be transformed by
making use of the change $x=\frac{\lambda}{4} z$ to the equality

$$z y''-d y' +\frac{y}{4}=0.$$ The solution of the last inequality,
by \cite[Formula~(7)~in~section~4.31,~p.~97]{watson}, is given by $$
y=z^{\frac 12 (d+1)}J_{d+1}(\sqrt
z)=\left(2\sqrt{x/\lambda}\right)^{1+d}J_{1+d}\left({2}\sqrt{x/\lambda}\right).$$

Then by \cite[Section~3.2,~p.~45]{watson} $$y'(x)
=\frac{1}{\sqrt{\lambda
x}}\left(2\sqrt{x/\lambda}\right)^{1+d}J_{d}\left({2}\sqrt{x/\lambda}\right)$$
and therefore

$$\lambda^{\frac{2+d}2} 2^{-d} y'(1)=J_{d}\left({2}\sqrt{1/\lambda}\right).$$

Thus the largest permissible value of $\lambda$ is $$\lambda^* =
\frac{4}{\alpha_d^2}$$ where $\alpha_d$ is the smallest positive
zero of $J_d$. 
Then as in the case $d\le 0$,  $y(x) > 0$ for $0 < x < 1$ since, by
Lema~\ref{air}, the smallest positive zero of $J_{d+1}$ is larger
than that of $J_d$. Finally we obtain
$$A_d=\frac{4}{\alpha_d^2}.$$
This finishes the proof of Lemma~\ref{malet}.
\end{proof}
\begin{proof}[Proof of Theorem~\ref{teorema}]
In view of comments after the statement of Theorem~\ref{boy}, the
inequality follows from Lemma~\ref{malet}. The equality statement
follows from the fact that $\alpha_0<\alpha_1<\dots<\alpha_d<\dots$,
Lemma~\ref{malet}, relation \eqref{possol} and
Proposition~\ref{boy}.
\end{proof}
\section{Refinement of $L^p$ norm}
We make use of the following interpolation theorem.
\begin{proposition}\cite{thorin} Let $T$ be a linear operator defined on a family $F$ of
functions that is dense in both $L^{p_1}$ and $L^{p_2}$ (for
example, the family of all simple functions). And assume that $Tf$
is in both $L^{p_1}$ and $L^{p_2}$ for any $f$ in $F$, and that $T$
is bounded in both norms. Then for any $p$ between ${p_1}$ and
${p_2}$ we have that $F$ is dense in $L^p$, that $Tf$ is in $L^p$
for any $f$ in $F$ and that $T$ is bounded in the $L^p$ norm. These
three ensure that $T$ can be extended to an operator from $L^p$ to
$L^p$.

In addition an inequality for the norms holds, namely for $t\in
(0,1)$ such that $$\frac 1p = \frac{1-t}{{p_1}}+\frac{t}{{p_2}}$$
there holds

    $$\|T\|_{L^p\to L^p}\le \|T\|^{1-t}_{L^{p_1}\to L^{p_1}}\cdot \|T\|^t_{L^{p_2}\to
L^{p_2}}.
$$
\end{proposition}

\begin{theorem}\label{rto}
For $1\le p\le 2$ we have
\begin{equation}\label{kad}\|\mathcal{C}_{\mathbf{U}}\|_{L^p\to L^p}\le
\frac{2}{\alpha^{2-2/p}} \text{ and }\|\mathcal{\bar
C}_{\mathbf{U}}\|_{L^p\to L^p}\le
\frac{2}{\alpha^{2-2/p}},\end{equation} and for $2\le p\le \infty$
we have
\begin{equation}\label{kadtad}\|\mathcal{C}_{\mathbf{U}}\|_{L^p\to L^p}\le
\frac{4}{3}\left(\frac{3}{2\alpha}\right)^{2/p} \text{ and
}\|\mathcal{\bar C}_{\mathbf{U}}\|_{L^p\to L^p}\le
\frac{4}{3}\left(\frac{3}{2\alpha}\right)^{2/p}.\end{equation} There
holds the equality in all inequalities in \eqref{kad} and
\eqref{kadtad} for $p=1$, $p=2$  and $ p =\infty$. Moreover for
$1\le p\le 2$ there holds the inequality
\begin{equation}\label{dri}\|\mathcal D_{\mathbf U}\|_{L^p\to L^p}\le 4 \alpha^{2/p-2},\end{equation} and if $2\le
p\le \infty$
\begin{equation}\label{drili}\|\mathcal D_{\mathbf U}\|_{L^p\to L^p} \le \frac{16}{3\pi}
\left(\frac{3\pi}{4\alpha}\right)^{2/p}.\end{equation}

\end{theorem}
\begin{proof}
Let $T$ be a linear operator defined by
$T=\mathcal{C}_{\mathbf{U}}\colon L^p(\mathbf{U},\Bbb C)\to
L^p(\mathbf{U}, \Bbb C)$. The inequalities follow by applying the
Riesz-Thorin theorem, Theorem~\ref{11}, Theorem~\ref{ok} and
Theorem~\ref{teorema} by taking $t =2-\frac 2p$, for $1=p_1\le p\le
p_2=2$, and $t=1-\frac 2p$ for $2=p_1\le p\le p_2=\infty$ to the
operator $T$ and observing that
$$2^{-1+\frac
2p}\left(\frac{2}{\alpha}\right)^{2-2/p}=\frac{2}{\alpha^{2-2/p}},\
1\le p\le 2
$$ and
$$\left(\frac{2}{\alpha}\right)^{1-(1-2/p)}\cdot \left(\frac
43\right)^{1-2/p}=\frac
43\cdot\left(\frac{3}{2\alpha}\right)^{2/p},\ \  2\le p\le \infty.$$

To prove \eqref{dri} and \eqref{drili} we take the operator
 $T=\mathcal{D}_\mathbf{U}:
L^p(\mathbf{U},\Bbb C)\to L^p(\mathbf{U}, \mathcal{M}_{2,2})$.

The inequalities \eqref{dri} and \eqref{drili} follow from
\eqref{th}, $\|\mathcal D_{\Bbb U}\|_p\le\|\mathcal{\bar C}_{\Bbb
U}\|_p+\|\mathcal C_{\Bbb U}\|_p$ and the relations

$$4^{-1+2/p} \left(\frac{4}{\alpha}\right)^{2-2/p}
=4 \alpha^{2/p-2},\ \ 1\le p\le 2$$
$$\left(\frac 4\alpha\right)^{2/p}\left(\frac
{16}{3\pi}\right)^{1-2/p}=\frac
{16}{3\pi}\left(\frac{3\pi}{4\alpha}\right)^{2/p}, 2\le p\le
\infty.$$
\end{proof}
By Riesz-Thorin theorem, the function $[0,1]\ni s\to \log \|\mathcal
C_{\mathbf U}\|_{L^{1/s}\to L^{1/s}}$ is convex and therefore
continuous. This, together with Theorem~\ref{rto} imply the fact
\begin{corollary} There are exactly two absolute constants $1<p_1<2$
and $2<p_2<\infty$ such that
$$\|\mathcal C_{\mathbf U}\|_{L^{p_1}\to L^{p_1}}=\|\mathcal
C_{\mathbf U}\|_{L^{p_2}\to L^{p_2}}=1.$$
\end{corollary}

As $\Delta u = 4 u_{z\overline z}$ we obtain the following result.
\begin{corollary}
Let $u\in W^{2,p}_0(\Bbb U)$, $1\le p\le \infty$ be a solution of
the Poisson equation $\Delta u = h$, $h\in L^p$. Then
$$\|\partial u\|_{L^p} \le\frac{\alpha^{2/p-2}}2
\|h\|_{L^p},\text { and }\|\bar\partial u\|_{L^p} \le
\frac{\alpha^{2/p-2}}2\|h\|_{L^p}, \ \ 1\le p\le 2 $$ and

$$\|\partial u\|_{L^p} \le\frac{1}{3}\left(\frac{3}{2\alpha}\right)^{2/p}
\|h\|_{L^p},\text { and }\|\bar\partial u\|_{L^p} \le
\frac{1}{3}\left(\frac{3}{2\alpha}\right)^{2/p} \|h\|_{L^p}, \ \
2\le p\le \infty .$$  The inequalities are sharp for $p\in \{1,
2,\infty\}$. Moreover  $$\|\nabla u\|_{L^p} \le
\alpha^{2/p-2}\|h\|_{L^p}, \ \ 1\le p\le 2$$ and

$$\|\nabla u\|_{L^p} \le\frac{4}{3\pi}
\left(\frac{3\pi}{4\alpha}\right)^{2/p}\|h\|_{L^p}, \ \ 2\le p\le
\infty,$$ with sharp constants for $p=1$ and $p=\infty$.
\end{corollary}

\end{document}